\documentclass[11pt]{article}
\usepackage[cp1251]{inputenc}
\usepackage[english]{babel}
\usepackage{amsmath, amsthm, amsfonts, amssymb}
    \usepackage[usenames]{color}
		
\usepackage[nottoc]{tocbibind}

\usepackage[backref]{hyperref} %for editing

\theoremstyle{plain}
\newtheorem{remark}{Remark}[section]

\newtheorem{example}{Example}

\newtheorem{theorem}{Theorem}[section]

\newtheorem{corollary}{Corollary}[section]

\renewcommand{\abstract}{\textbf{Abstract. }}

\makeatletter
\@addtoreset{equation}{section}
\usepackage{geometry}
\allowdisplaybreaks
\makeatother

\begin{document}

\title{\bf On optimal recovery of unbounded operators \\
from inaccurate data}
%in Hilbert and Banach spaces
%by  truncation method}

%\maketitle
%\makeatletter
%\let\@fnsymbol\@arabic
%\makeatother

%\author{}
\author{Oleg Davydov\thanks{Department of Mathematics,
	University of Giessen, Arndtstrasse 2, 35392 Giessen, Germany,
	\url{oleg.davydov@math.uni-giessen.de}}\and
	Sergei Solodky\thanks{Institute of Mathematics, National Academy of Sciences of Ukraine,
    Tereshchenkivska Str. 3, 01601, Kyiv, Ukraine, and
    Department of Mathematics,
	University of Giessen, Arndtstrasse 2, 35392 Giessen, Germany,
    \url{solodky@imath.kiev.ua}} }

%\date{\today}

\date{May 11, 2025}

\maketitle

\begin{abstract}
The problems of optimal recovery of unbounded operators are studied.
Optimality means the highest possible accuracy and the minimal amount of discrete information involved.
It is established that the truncation method, when certain conditions are met, realizes the optimal values of the studied quantities.
As an illustration of the general results, problems of numerical differentiation and the backward parabolic equation are considered.
\end{abstract}

%\subjclass{47A52, 65R30.}

\section{Introduction}
This study is devoted to the optimization of approximate methods for the recovery of values
of unbounded operators in the case of inaccurate input data.
It is well known that such problems are unstable and require regularization when constructing approximations.
It is quite natural that optimization was first studied for well-posed problems,
in particular, in approximation theory.
To date, many optimal methods have been developed for such problems
(see, for example, \cite{MR}, \cite{LGM}, \cite{Tikh}).
Separately, we should highlight the theory of information complexity \cite{TWW}, \cite{TW}, \cite{NW1}, \cite{NW2}, \cite{NW3},
where complexity is understood as the smallest amount of discrete information necessary
to solve a problem with a given accuracy.
As it turned out when solving well-posed problems, introducing noise
into the input data does not lead to fundamental changes.
Therefore, for such problems, it is not customary to consider models with perturbed input information.

As for ill-posed problems, the issue of optimizing approximate methods
for solving them has been studied much less.
Here, article \cite{St} can be considered the starting point, where the optimal recovery
of unbounded operators by bounded operators was studied (for more details, see \cite{Ar96}).
As it turned out later, it is convenient to evaluate the accuracy of solving ill-posed problems
through the error of the input data, using the smoothness of the desired solution,
expressed using the so-called source condition.
Among the works in this area we can highlight \cite{Bak78}, \cite{VV86},
where moderately ill-posed problems with the source condition in the form of a power function were considered.
These studies were later generalized to different types of ill-posed problems in \cite{NPT}.
The first results on the information complexity of ill-posed problems were related
to the minimal radius of Galerkin information (see, for example, \cite{PS}, \cite{MS}),
where sharp estimates were found for the optimal accuracy of solving Fredholm integral equations of the first kind,
if discrete information is understood as the perturbed values of the Fourier coefficients of the operator kernel
and the right-hand side of the problem.
The most general results on the information complexity of ill-posed problems in Hilbert spaces
were obtained in \cite{MP}, where optimization was carried out using various algorithms dealing
with values of linear functionals as discrete information.

The research of this study can be considered as a continuation of \cite{MP}
for algorithms dealing with values of continuous (but not necessarily linear) functionals.
In addition, we will study the case when the approximation error is estimated in the metrics of Banach spaces.

It should be note that to solve the problem optimal recovery, we propose a truncation method.
At present, this method is being intensively studied by many researchers and has proven itself well in solving various types of
ill-posed problems (see, for example, \cite{Mathe&Pereverzev_2002_JAT}, \cite{Qian&Fu&Xiong&Wei_2006}, \cite{Zhao_2010}, \cite{LNP},
\cite{Solodky&Sharipov_2015_P116_124}, \cite{SSS_CMAM}).
In the truncation method the stability to small perturbations in input data and the required accuracy level are achieved by properly choosing the discretization parameter, which serves as a regularization parameter here.

The article is organized as follows. In Section \ref{probset},
the problem statement for optimal recovery of unbounded operators in a Hilbert space is given.
Section \ref{below} establishes lower bounds for the error of optimal recovery.
In Section \ref{ertru}, the accuracy estimates for the truncation method were established.
Section \ref{opttru} is devoted to discussing the optimality of the truncation method.
Section \ref{banach} considers the problem of optimal recovery in the metric of Banach spaces.
Section \ref{ND&BPE} studies the problems of optimal recovery in the numerical differentiation and the backward
parabolic equation.

\section{Problem setting} \label{probset}
Let us proceed to the problem setting.
Let $H$ be a real Hilbert space and $A: D(A)\to H$ an unbounded positive self-adjoint linear operator with eigenvalues
$\{\mu_k\}_{k=0}^\infty$  ordered such that
\begin{equation}  \label{mu}
0 < \mu_0 \le \mu_1 \le \ldots \le \mu_k \le \ldots , \qquad \lim_{k\to \infty} \mu_k = \infty ,
\end{equation}
and let $\{w_k\}_{k=0}^\infty$ be the corresponding orthonormal eigenbasis. Then for any $f\in D (A)$ we have
$$
f = \sum_{k=0}^\infty \langle f , w_{k} \rangle w_k,\qquad
A f = \sum_{k=0}^\infty \mu_k \langle f , w_{k} \rangle w_k,
$$
with
$$
D (A) = \Big\{ f\in H: \sum_{k=0}^\infty \mu_k^2 \langle f , w_{k} \rangle^2 < \infty\Big\} \neq H .
$$
In addition, we consider a subset $W$ of $D(A)$ defined by
\begin{equation} \label{W}
W = \Big\{ f\in H: \sum_{k=0}^\infty \xi_k^2 \langle f , w_{k} \rangle^2 < \infty\Big\} ,
\end{equation}
where $\{\xi_k\}_{k=0}^\infty$ is a nondecreasing sequence such that
\begin{equation}  \label{xi}
0 < \xi_0 \le \xi_1 \le \ldots \le \xi_k \le \ldots , \qquad \lim_{k\to \infty} \xi_k = \infty ,
\end{equation}
and
%the sequence $\{\frac{\mu_k}{\xi_k}\}$ is monotone increasing and
\begin{equation} \label{lim}
\lim_{k\to \infty} \frac{\mu_k}{\xi_k} = 0 .
\end{equation}
Then $W$ is a Hilbert space with inner product
$$
\langle f , g \rangle_{_W} := \sum_{k=0}^\infty \xi_k^2 \langle f , w_{k}\rangle \langle g , w_{k} \rangle
$$
and corresponding norm $\|f\|_{_W} = \langle f , f\rangle_W^{1/2}$.

Assume that instead of $f$ we are given inaccurate information in the form of $f^\delta\in H$ with
$\|f-f^\delta\|_{H}\leq \delta$, and we are looking for a mapping $\Phi$ that recovers $Af$ using $f^\delta$ with the highest
possible accuracy.

Let $G: H\to {\mathbb R}^N$ be a continuous information map and $\Psi: {\mathbb R}^N\to H$ a recovery algorithm.
Given any input error level $\delta>0$, we consider the worst case error of the recovery of $Af$ by the algorithm $\Psi$
from the information vector $G(f^\delta)$, where $f^\delta\in H$ represents a $\delta$-pertubation of $f\in W$,
$$
\varepsilon_\delta (A, W, G, \Psi)_H :=
\sup_{\substack{f\in W , \\ \|f\|_{_W}\leq 1}}\quad
\sup_{\substack{f^\delta\in H , \\ \|f-f^\delta\|_{_H}\leq \delta}}%\quad
\|Af - \Psi(G(f^\delta))\|_H
$$
and the %error of optimal recovery from information of cardinality $N$, which we call the
\emph{$(N,\delta)$-optimal recovery error},
$$
R_{N,\delta} (A, W)_H :=
\inf_{\substack{G: H\to {\mathbb R}^N \\ \rm{(continuous\, map)}}}\quad
\inf_{\substack{\Psi: {\mathbb R}^N\to H \\ \rm{(arbitrary\, map)}}}%\quad
\varepsilon_\delta (A, W, G, \Psi)_H .
$$

For any given error level $\delta>0$, the optimal recovery error cannot get arbirarily small when $N\to\infty$. Therefore we
also consider the \emph{$\delta$-optimal recovery error}
$$
R_\delta (A, W)_H := \inf_N R_{N,\delta} (A, W)_H.
$$

We provide two typical examples of operators $A$ for which the problem of optimal recovery is of interest.

\begin{example}\rm  \label{ex1}
Let $H$ be the Hilbert space $L_{2\pi,0}^2$ of $2\pi$-periodic square integrable functions with zero mean.
Let $A_1$ be the operator
\begin{equation}  \label{prex1}
A_1f = - f'' .
\end{equation}
Then
$$
D (A_1) = \Big\{ f\in L_{2\pi,0}^2: \sum_{k=0}^\infty \mu_k^2 \langle f , w_{k} \rangle^2 < \infty\Big\}
$$
with
$$
w_{2k} = \cos (k+1)x,\quad
w_{2k+1} = \sin (k+1)x,\quad k=0,1,2,\ldots,
$$
and
$$
\mu_{2k}=\mu_{2k+1} = (k+1)^2,\quad k=0,1,2,\ldots.
$$
Consider the problem of numerical differentiation formulated as optimal recovery of $A_1f$ for $f\in W$, where
$$
W=\Big\{ f\in L_{2\pi,0}^2: \sum_{k=0}^\infty \underline{k}^{2\gamma} \langle f , w_{k} \rangle^2 < \infty\Big\}
$$
with some $\gamma>2$ and $\underline{k}:=\max\{1,k\}$,\ $k=0,1,2,\dots$.
Then $W$ has the form \eqref{W} with $\xi_k=\underline{k}^\gamma$ that satisfies \eqref{xi} and \eqref{lim}.
\end{example}

\begin{example}\rm  \label{ex2}
Let $H$ be an arbitrary Hilbert space and $L$ an unbounded self-adjoint operator $L: D(L)\to H$
with eigenvalues $\lambda_k$ and eigenbasis $\{w_k\}_{k=0}^\infty$ of $H$,
$0 \le \lambda_0 \le \lambda_1 \le \cdots \le \lambda_k \le \cdots$, $\lambda_k\to \infty$. Let
$$
A_2 f = \sum_{k=0}^\infty \langle f , w_{k} \rangle e^{\lambda_k t} w_k ,
$$
where $t>0$ is fixed.
That is, $A_2f$ is the solution $u(t)$ of the backward parabolic equation
\begin{equation}  \label{prex2}
\frac{d u}{d t} - L u = 0, \qquad u(0) = f .
\end{equation}
Hence, $\mu_k=e^{\lambda_k t}$. We may choose $\xi_k=\underline{k}^\mu e^{\lambda_k T}$ for some $\mu\ge 0$, $T\ge t$.
\end{example}

%%%%%%%%%%%%%%%%%%%%%%%%%
\section{Estimate of optimal recovery error from below}\label{below}

For any $\delta>0$, we set
\begin{equation}  \label{N_delta}
N_\delta:=\min\{N\ge 0:\xi_N \ge 1/\delta\} .
\end{equation}

\begin{theorem}  \label{th1}
For any $\delta$,
\begin{equation} \label{R_delta}
R_\delta (A, W)_H \ge \frac{\mu_{N_\delta}}{\xi_{N_\delta}}.
\end{equation}
Moreover, for any $\delta$ and $N$,
\begin{equation} \label{est_bel}
R_{N,\delta} (A, W)_H \ge \max \Big\{ \frac{\mu_{N_\delta}}{\xi_{N_\delta}},\;
\min_{0\le k\le N}
\frac{\mu_k}{\xi_k}\Big\} .
\end{equation}
\end{theorem}

\begin{proof} We first choose
$$
f_* = \xi_{N_\delta}^{-1} w_{N_\delta} .
$$
Then $f_*\in W$, with
$$
\|f_*\|_{_W} = 1,\quad \|f_*\|_{_H} = \xi_{N_\delta}^{-1},\quad \|Af_*\|_{_H} = \frac{\mu_{N_\delta}}{\xi_{N_\delta}} .
$$
Since $\|f_*\|_H \le \delta$, the element $f_*^\delta=0$ is a $\delta$-perturbation of both $f_*$ and $-f_*$. Given
$\Phi:H\to H$ (no linearity, no continuity assumed),
let $g:= \Phi(0) = \Phi(f_*^\delta) = \Phi(-f_*^\delta)$.

We have
$$
2 \|Af_*\|_H = \|A(-f_*) - Af_*\|_H \le \|A(-f_*) - g\|_H + \|Af_* - g\|_H .
$$
Hence
$$
\sup_{\substack{f\in W , \\ \|f\|_{_W}\leq 1}}\quad
\sup_{\substack{f^\delta\in H , \\ \|f-f^\delta\|_{_H}\leq \delta}}%\quad
\|Af - \Phi(f^\delta)\|_H \ge \min\big\{ \|A(-f_*) - g\|_H, \|Af_* - g\|_H\big\} \ge \|Af_*\|_H
= \frac{\mu_{N_\delta}}{\xi_{N_\delta}},
$$
and \eqref{R_delta} follows.

For the second part consider the $N$-dimensional sphere
$$
S^N := \Big \{ x\in {\mathbb R}^{N+1} : \sum_{k=0}^{N} x_k^2  = 1\Big \},
$$
and for any $x\in S^N$ define
$$
f_x = \sum_{k=0}^{N} \frac{x_k}{\xi_k} w_k \in W,
$$
with
$$
\|f_x\|_W^2 = \sum_{k=0}^{N} \frac{x_k^2}{\xi_k^2} \xi_k^2 = 1 .
$$
For any continuous map $G: H\to {\mathbb R}^{N}$ define continuous $F: S^N\to {\mathbb R}^{N}$
by $F(x) := G(f_x)$.
By Borsuk's antipodal theorem there exists $x\in S^N$ such that
$F(x)=F(-x)$, that is
$$
G(f_x) = G(-f_x) = G(f_{-x}) .
$$
Let $f_x^\delta:=f_x$, $f_{-x}^\delta:=f_{-x}$. For any $\Psi: {\mathbb R}^{N}\to H$, set
$$
g :=  \Psi(G(f_x^\delta)) = \Psi(G(f_x)) = \Psi(G(f_{-x})) = \Psi(G(f_{-x}^\delta)) .
$$
Then $Af_x = -Af_{-x}$ and hence
$$
2 \|Af_x\|_H = \|Af_x - Af_{-x}\|_H \le \|Af_x - g\|_H + \|Af_{-x} - g\|_H ,
$$
which implies
$$
\varepsilon_\delta (A, W, G, \Psi)_H^2 \ge \|Af_x\|_H^2
= \sum_{k=0}^{N} \frac{\mu_k^2}{\xi_k^2} x_k^2  \ge
\min_{0\le k\le N} \frac{\mu_k^2}{\xi_k^2} \; \sum_{k=0}^{N} x_k^2
= \min_{0\le k\le N} \frac{\mu_k^2}{\xi_k^2} .\qedhere
$$
% %alternative with {align*}, but not needed here
% \begin{align*}
% \varepsilon_\delta (A, W, G, \Psi)_H^2 \ge \|Af_x\|_H^2
% &= \sum_{k=0}^N \mu_k^2 x_k^2\\
% &= \sum_{k=0}^N \frac{\mu_k^2}{\xi_k^2} x_k^2 \xi_k^2 \ge
% \min_{k\le N} \frac{\mu_k^2}{\xi_k^2} \, \sum_{k=0}^N x_k^2 \xi_k^2
% = \min_{k\le N} \frac{\mu_k^2}{\xi_k^2} .\qquad\qquad\qedhere
% \end{align*}

\end{proof}

%%%%%%%%%%%%%%%%%%%%%%%%%%%%%%%%%%%
\section{Recovery error of truncation method}\label{ertru}

We now give a formula for the recovery error of the \emph{truncation method} defined by the information map $\tilde G_N$ and
recovery algorithm $\tilde \Psi_N$, where
\begin{equation}  \label{tm1}
\tilde G_N(f) := [\langle f , w_{k} \rangle]_{k=0}^{N-1} , \qquad
\tilde \Psi_N(x) := \sum_{k=0}^{N-1} \mu_k x_k w_k .
\end{equation}
Hence,
\begin{equation}  \label{tm2}
\tilde \Psi_N(\tilde G_N(f^\delta)) = AS_N(f^\delta) =
\sum_{k=0}^{N-1} \mu_k\langle f^\delta , w_{k} \rangle  w_k ,
\end{equation}
where $S_N(f)$ %:=\sum_{k=0}^{N-1} \langle f, w_{k} \rangle  w_k$
denotes the $N$-th partial sum of the series
$f=\sum_{k=0}^\infty \langle f, w_{k} \rangle  w_k$.

The following statement contains an accuracy estimate for the truncation method (\ref{tm1}).

\begin{theorem}  \label{th2}
For any $\delta$ and $N$ the recovery error of the truncation method (\ref{tm1}) is given by
\begin{equation} \label{ep_eq}
\varepsilon_\delta (A, W, \tilde G_N, \tilde \Psi_N)_H =
\Big(\max_{k\ge N} \frac{\mu_k^2}{\xi_k^2} + \delta^2 \mu_{N-1}^2\Big)^{1/2} .
\end{equation}
Hence,
\begin{equation}\label{tdo}
\inf_N\varepsilon_\delta (A, W, \tilde G_N, \tilde \Psi_N)_H
\le \varepsilon_\delta (A, W, \tilde G_{N_\delta}, \tilde \Psi_{N_\delta})_H <
\Big(\frac{\mu_{N_{\delta-1}}^2}{\xi_{N_{\delta-1}}^2} + \max_{k\ge N_\delta} \frac{\mu_k^2}{\xi_k^2}\Big)^{1/2} .
\end{equation}
\end{theorem}

\begin{proof}
We have for any $f\in W$ with $\|f\|_W \le 1$ and $f^\delta\in H$ with $\|f-f^\delta\|_H\le\delta$,
$$
\|Af - \tilde\Psi_N(\tilde G_N(f^\delta))\|^2_H = \|Af - S_N(Af^\delta)\|_H^2 = \sigma_1 + \sigma_2 ,
$$
where
$$
\sigma_1 := \Big\|Af - \sum_{k=0}^{N-1} \mu_k\langle f , w_k\rangle  w_k\Big\|_H^2 =
\Big\|\sum_{k=N}^\infty \mu_k\langle f , w_k\rangle  w_k\Big\|_H^2 = \sum_{k=N}^\infty \mu_k^2\langle f , w_k\rangle^2,
$$
$$
\sigma_2 := \Big\|\sum_{k=0}^{N-1} \mu_k\langle f - f^\delta , w_k\rangle  w_k\Big\|_H^2 = \sum_{k=0}^{N-1} \mu_k^2\langle f -
f^\delta, w_k\rangle^2.
$$
Now
$$
\sigma_1
%=\sum_{k=N}^\infty  \xi_k^2 \frac{\mu_k^2}{\xi_k^2}\langle f , w_k\rangle^2
\le \|f\|^2_W \max_{k\ge N} \frac{\mu_k^2}{\xi_k^2}
\le \max_{k\ge N} \frac{\mu_k^2}{\xi_k^2} ,
$$
and
$$
\sigma_2
\le \delta^2 \max_{k\le N-1} \mu_k^2 = \delta^2 \mu_{N-1}^2 ,
$$
which implies
$$
\varepsilon_\delta (A, W, \tilde G_N, \tilde \Psi_N)_H \le
\Big(\max_{k\ge N} \frac{\mu_k^2}{\xi_k^2} + \delta^2 \mu_{N-1}^2\Big)^{1/2}.$$

To show the opposite inequality in \eqref{ep_eq} we choose $\ell\ge N$ such that
$$
\frac{\mu_\ell^2}{\xi_\ell^2} = \max_{k\ge N} \frac{\mu_k^2}{\xi_k^2} .
$$
Consider $f_2=\xi_\ell^{-1} w_\ell$, $f_2^\delta=f_2+\delta w_{N-1}$.
Then $f_2\in W$, $\|f_2\|_W=1$, $\|f_2-f_2^\delta\|_H=\delta$,
hence,
$$
\varepsilon_\delta (A, W, \tilde G_N, \tilde \Psi_N)_H^2 \ge
\|Af_2 - S_N(Af_2^\delta)\|_H^2 = \sigma_1 + \sigma_2
= \frac{\mu_\ell^2}{\xi_\ell^2} + \delta^2 \mu_{N-1}^2 .
$$

The estimate \eqref{tdo} follows because $\delta<\xi_{N_{\delta}-1}^{-1}$ by \eqref{N_delta}.
\end{proof}

%%%%%%%%%%%%%%%%%%%%%%%%%%%%%%%%%%%%5
\section{Optimality of truncation method}\label{opttru}
As is known, problems with unbounded operators are unstable in the case of inaccurate input data.
In other words, for such problems, small perturbations of the input data may lead to a large divergence
between the solutions of the exact and perturbed problems.
In such cases, regularization techniques are required to ensure the stability of computations.
Note the process of regularization in (\ref{tm1}) consists of choosing the discretization parameter $N$ depending
on the error value $\delta$ and parameters $\mu_k, \xi_k$ such that to minimize the error of the method.
Thus, regularization allows not only the convergence of approximations to be achieved but also
guarantees the optimal solution accuracy.
The results of this section are devoted to the study of the optimality of the truncation method.

Theorems~\ref{th1} and \ref{th2} imply the following double estimates  of the optimal and $\delta$-optimal recovery errors,
where the upper bounds are obtained with the help of the truncation method.

\begin{corollary} For any $\delta$,
\begin{equation} \label{dobd}
\frac{\mu_{N_\delta}}{\xi_{N_\delta}}
\le R_{\delta} (A, W)_H \le R_{N_\delta,\delta} (A, W)_H <
\Big(\frac{\mu_{N_{\delta}-1}^2}{\xi_{N_{\delta}-1}^2} + \max_{k\ge N_\delta} \frac{\mu_k^2}{\xi_k^2}\Big)^{1/2},
\end{equation}
and for any $\delta$ and $N$,
\begin{equation} \label{dob}
\max \Big\{ \frac{\mu_{N_\delta}}{\xi_{N_\delta}},\;
\min_{k\le N} \frac{\mu_k}{\xi_k}\Big\}
\le R_{N,\delta} (A, W)_H \le
\Big(\max_{k\ge N} \frac{\mu_k^2}{\xi_k^2} + \delta^2 \mu_{N-1}^2\Big)^{1/2}.
\end{equation}
\end{corollary}

We now investigate under which conditions the lower and upper bounds in \eqref{dobd} and \eqref{dob}
only differ by a constant factor, so that the truncation method is optimal in order.

For the sake of simplicity we assume from now on that the sequence $\{\frac{\mu_N}{\xi_N}\}$ is monotone
decreasing, that is
\begin{equation}  \label{muxi}
\frac{\mu_{0}}{\xi_{0}} \ge\frac{\mu_{1}}{\xi_{1}} \ge  \cdots \ge \frac{\mu_{k}}{\xi_{k}}  \ge \cdots ,
\qquad \lim_{k\to \infty} \frac{\mu_{k}}{\xi_{k}} = 0 .
\end{equation}
Then the formulas \eqref{ep_eq}--\eqref{dob} simplify to
\begin{align}
\label{dobd1}
\frac{\mu_{N_\delta}}{\xi_{N_\delta}}
\le R_{\delta} (A, W)_H
&\le R_{\delta,N_\delta} (A, W)_H
\le \varepsilon_\delta (A, W, \tilde G_{N_\delta}, \tilde \Psi_{N_\delta})_H
<\Big(\frac{\mu_{N_{\delta}-1}^2}{\xi_{N_{\delta}-1}^2} + \frac{\mu_{N_{\delta}}^2}{\xi_{N_{\delta}}^2}\Big)^{1/2},
\\
\label{dob1}
\max \Big\{ \frac{\mu_{N_\delta}}{\xi_{N_\delta}},\;\frac{\mu_N}{\xi_N}\Big\}
&\le R_{N,\delta} (A, W)_H \le \varepsilon_\delta (A, W, \tilde G_N, \tilde \Psi_N)_H =
\Big( \delta^2 \mu_{N-1}^2 +\frac{\mu_{N}^2}{\xi_{N}^2}\Big)^{1/2}.
\end{align}

In what follows, the same symbol $c$ will denote various positive constants that do not depend on $\delta$ and $N$.

We want to establish inequalities of the type
\begin{equation} \label{est_up}
\varepsilon_\delta (A, W, \tilde G_N, \tilde \Psi_N)_H \le
c R_{N,\delta} (A, W)_H
\end{equation}
for some constants $c\ge 1$.

Let us first consider the case $N\le N_\delta$.
Assume $\{\frac{\mu_N}{\xi_N}\}$ is monotone decreasing.
 By (\ref{est_bel}) we get
$R_{N,\delta} (A, W)_H \ge \frac{\mu_N}{\xi_N}$. Then (\ref{est_up}) is satisfied in view of (\ref{ep_eq}) if
$$
\frac{\mu_N^2}{\xi_N^2} + \delta^2 \mu_{N-1}^2 \le c^2 \frac{\mu_N^2}{\xi_N^2} ,
$$
that is
$$
\delta \mu_{N-1} \le (c^2-1)^{1/2} \frac{\mu_N}{\xi_N} .
$$
If $N<N_\delta$, then $\delta < \xi_N^{-1}$ by (\ref{N_delta}), and hence the condition is satisfied
as soon as $(c^2-1)^{1/2} \ge 1$, that is $c \ge \sqrt{2}$.
Hence we obtained the following statement.

\begin{theorem}  \label{th3}
If  $\{\frac{\mu_N}{\xi_N}\}$ is monotone decreasing,
then for any $N < N_\delta$,
\begin{equation} \label{ep_R}
\varepsilon_\delta (A, W, \tilde G_N, \tilde \Psi_N)_H \le
\sqrt{2} \frac{\mu_N}{\xi_N} \le \sqrt{2}
R_{N,\delta} (A, W)_H .
\end{equation}
\end{theorem}

For $N=N_\delta$ we have $\delta < \xi_{N_\delta-1}^{-1}$ by (\ref{N_delta}), hence
a sufficient condition for $c$ in (\ref{est_up}) becomes
$$
\frac{\mu_{N_\delta}^2}{\xi_{N_\delta}^2} + \frac{\mu_{N_\delta-1}^2}{\xi_{N_\delta-1}^2} \le
c^2 \frac{\mu_{N_\delta}^2}{\xi_{N_\delta}^2} ,
$$
that is
$$
c \ge \Big(1+ \Big(\frac{\xi_{N_\delta}}{\mu_{N_\delta}}
\frac{\mu_{N_\delta-1}}{\xi_{N_\delta-1}}\Big)^2\Big)^{1/2} ,
$$
and we get

\begin{theorem}  \label{th4}
If $\{\frac{\mu_N}{\xi_N}\}$ is monotone decreasing, then
\begin{equation} \label{ep_c}
\varepsilon_\delta (A, W, \tilde G_{N_\delta}, \tilde \Psi_{N_\delta})_H \le
\Big(\frac{\mu_{N_\delta}^2}{\xi_{N_\delta}^2} +
\frac{\mu_{N_\delta-1}^2}{\xi_{N_\delta-1}^2}\Big)^{1/2} \le
K_\delta \frac{\mu_{N_\delta}}{\xi_{N_\delta}} ,
\end{equation}
where
\begin{equation} \label{c_delta}
K_\delta := \Big(1+ \Big(\frac{\xi_{N_\delta}}{\mu_{N_\delta}}
\frac{\mu_{N_\delta-1}}{\xi_{N_\delta-1}}\Big)^2\Big)^{1/2}.
\end{equation}
\end{theorem}

\begin{remark} {\rm In Example \ref{ex1} we have for any $n=1,2,\ldots$:
\begin{flushleft}
$\frac{\mu_{2n}}{\xi_{2n}} = \frac{(n+1)^2}{(2n)^{\gamma}} = \frac{(1+1/n)^2}{4(2n)^{\gamma-2}} , \quad
\frac{\mu_{2n+1}}{\xi_{2n+1}} = \frac{(n+1)^2}{(2n+1)^{\gamma}} = \frac{\left(1+\frac{1}{2n+1}\right)^2}{4(2n+1)^{\gamma-2}}$ ,
\end{flushleft}
\begin{flushleft}
$$
\frac{\xi_{N_\delta}}{\mu_{N_\delta}} \frac{\mu_{N_\delta-1}}{\xi_{N_\delta-1}} =
\left\{ \begin{array}{cl}
\frac{(2n)^{\gamma-2}(1+\frac{1}{2n-1})^2}{(1+1/n)^2(2n-1)^{\gamma-2}}, \ & {\rm if}\ N_\delta=2n {\rm(even)},\\\\
\frac{(2n+1)^{\gamma-2}(1+1/n)^2}{(1+\frac{1}{2n+1})^2 (2n)^{\gamma-2}}, \ & {\rm if} N_\delta=2n+1 {\rm(odd)}
\end{array} \right. \le
\Big(\frac{N_\delta}{N_\delta-1}\Big)^\gamma \ {\rm is\ bounded, and}\ K_\delta \ {\rm bounded}.
$$
\end{flushleft}
\begin{flushleft}
For Example \ref{ex2}:
$\frac{\mu_N}{\xi_N}= N^{-\mu} e^{-\lambda_N(T-t)}$, hence
$$
\frac{\xi_{N_\delta}}{\mu_{N_\delta}} \frac{\mu_{N_\delta-1}}{\xi_{N_\delta-1}} =
\Big(\frac{N_\delta}{N_\delta-1}\Big)^\mu e^{(\lambda_{N_\delta}-\lambda_{N_{\delta}-1})(T-t)} ,
$$
\end{flushleft}
which is bounded if $t=T$.
}
\end{remark}

Now let us return to the discussion of optimality and consider the case $N> N_\delta$.\\
% (We again included $N=N_\delta$.)\\
By (\ref{ep_eq}) we have
$$
\varepsilon_\delta (A, W, \tilde G_N, \tilde \Psi_N)_H \ge \delta \mu_{N-1} .
$$
If $\{\frac{\mu_N}{\xi_N}\}$ is monotone decreasing, then by using (\ref{R_delta})
we get for all $N>N_\delta$
$$
\varepsilon_\delta (A, W, \tilde G_N, \tilde \Psi_N)_H \ge \frac{1}{\sqrt{2}}
\Big(\frac{\mu_{N_\delta}^2}{\xi_{N_\delta}^2} + \delta^2 \mu_{N_\delta-1}^2\Big)^{1/2} ,
$$
which implies the following statement.

\begin{theorem}  \label{th6}
If $\{\frac{\mu_N}{\xi_N}\}$ is monotone decreasing, then for any $N>N_\delta$,
\begin{equation} \label{ep_R1}
\varepsilon_\delta (A, W, \tilde G_N, \tilde \Psi_N)_H \ge
\frac{1}{\sqrt{2}}
\varepsilon_\delta (A, W, \tilde G_{N_\delta}, \tilde \Psi_{N_\delta})_H
\ge \frac{1}{\sqrt{2}} R_{N_\delta,\delta} (A, W)_H .
\end{equation}
\end{theorem}

This shows that we cannot get much better accuracy by using the truncation method with $N>N_\delta$.

To summarize, we have established the optimality of the choice $N=N_\delta$ in the sense of the following theorem obtained
 from (\ref{ep_c}) and (\ref{R_delta}).

\begin{theorem}  \label{th7}
If $\{\frac{\mu_N}{\xi_N}\}$ is monotone decreasing, then
%for any $N>N_\delta$ it holds true
\begin{equation} \label{ep_ineq}
\frac{\mu_{N_\delta}}{\xi_{N_\delta}}
\le R_\delta(A, W)_H \le R_{N_\delta,\delta}(A, W)_H \le
\varepsilon_\delta (A, W, \tilde G_{N_\delta}, \tilde \Psi_{N_\delta})_H \le
\Big(\frac{\mu_{N_\delta}^2}{\xi_{N_\delta}^2} + \frac{\mu_{N_\delta-1}^2}{\xi_{N_\delta-1}^2}\Big)^{1/2}  .
\end{equation}

\end{theorem}

\vskip 2mm

Since we know that in certain cases the lower and upper bounds in (\ref{ep_ineq}) my be of different order,
it may be useful to consider specific error levels $\delta$ such that in (\ref{N_delta})
$$
\delta = \xi_{N_\delta}^{-1} .
$$
In this case, Theorem \ref{th3} applies to $N=N_\delta$ as well, that is, we have

\begin{theorem}  \label{th8}
If $\delta = \xi_{N_\delta}^{-1}$, then
\begin{equation} \label{ep_mu}
\frac{\mu_{N_\delta}}{\xi_{N_\delta}} \le
R_\delta(A, W)_H \le  R_{{N_\delta},\delta}(A, W)_H \le
\varepsilon_\delta (A, W, \tilde G_{N_\delta}, \tilde \Psi_{N_\delta})_H \le
\sqrt{2} \frac{\mu_{N_\delta}}{\xi_{N_\delta}} .  \nonumber
\end{equation}
Moreover, if $\delta \le c \xi_{N_\delta}^{-1}$ for a constant $c\ge 1$,
then the same inequalities hold, with the right hand side replaced by
$(1+c^2)^{1/2} \frac{\mu_{N_\delta}}{\xi_{N_\delta}}$.
\end{theorem}

\vskip 2mm

As usual, in order to express the equivalence of a pair of expressions $E_1,E_2$ in order, we write
$$
E_1\asymp E_2$$
as long as  there are constants $c',c''>0$ independent of certain parameters,  such that
$$
c' E_2\le E_1\le c'' E_2.$$

We now explore the behaviour of $R_{N,\delta} (A, W)_H$ for $N$ chosen depending on $\delta$ such that
$$
\frac{1}{\xi_N}\asymp \delta.$$

\begin{remark} {\rm For a rapidly growing sequence $\xi_N$, the indices $N=N(\delta)$ satisfying the required double inequality
$$
\frac{c'}{\delta} \le \xi_{N} \le \frac{c''}{\delta}$$
may not exist.
For example, let $\xi_N=e^{\alpha N^\beta}$. Then we get inequalities
$$
\ln c' + \log \frac{1}{\delta} \le \alpha N^\beta \le \ln c'' + \log \frac{1}{\delta} ,
$$
that is
$$
\alpha N^\beta \in \log \frac{1}{\delta} + [\ln c', \ln c''] ,
$$
interval of a fixed length $\ln c'' - \ln c'$.
However, for $\beta>1$ the distances between the elements of the sequence $\alpha N^\beta$ grow
and they will become greater than $\ln c'' - \ln c'$ for sufficiently big $N$ needed to get
to the sizes of $\log 1/\delta$ for a small $\delta$.
In such cases assumption $\delta \asymp \xi_{N}^{-1}$ is invalid for sufficiently small $\delta>0$.
}
\end{remark}

Let
\begin{equation} \label{slowly}
\xi_k = c k^\eta e^{\alpha k^\beta},\qquad\text{with}
\quad c>0,\quad 0\le \beta\le 1,\quad \eta, \alpha\ge 0,\quad \alpha+\eta>0 ,
\end{equation}
where the requirement $\beta\le 1$ ensures that the assumption $\delta \asymp \xi_{N}^{-1}$ makes sense.

From the results established above we can obtain, in particular, the following statement.
\vskip 2mm

\begin{theorem}  \label{th9}
Let $\xi_k$ be given by (\ref{slowly}). For any $\delta>0$ and $N=N(\delta)$ chosen such that $\frac{1}{\xi_N}\asymp \delta$,
we have
\begin{equation} \label{optC}
R_{\delta} (A, W)_H \asymp
R_{N,\delta} (A, W)_H \asymp \frac{\mu_N}{\xi_N} ,
\end{equation}
which indicates that order-optimal estimates are achieved by the truncation method (\ref{tm1}), (\ref{tm2}).
\end{theorem}

\vskip 2mm

\begin{remark}  {\rm
Let's compare the results of this section with the results from \cite{MP}.
It is easy to show that the estimates for the optimal recovery of ill-posed problems in
\cite{MP} and Theorem \ref{th9} of this study coincide in order.
This is a common point in the results discussed. But there is an important difference between them.
It lies in the various types of functionals used by the considered algorithms.
More precisely, linear functionals were used in \cite{MP}, while we used continuous functionals.
This difference is explained by different approaches to estimating from below the value of the best approximation.
Namely, in \cite{MP} the corresponding estimate was obtained using the Gelfand width,
which requires the use of linear functionals,
and in our study the lower estimate follows from Borsuk's theorem,
which deals with continuous functionals.
Since the estimates of optimal recovery are the same in both compared works,
Therefore, the results of this Section can be considered as a natural complement to the studies of \cite{MP}.
}
\end{remark}

\section{Optimal recovery in metrics of Banach spaces} \label{banach}
It should be noted that, in contrast to Hilbert spaces, the question of optimal recovery
of solutions to ill-posed problems in the metrics of other spaces has been studied much less.
Therefore it makes sense to study the quantity $R_{N,\delta}$
in the metrics of Banach spaces.

To this end, we present the following concepts and definitions.
Let $Q$ be a compact subset of the Euclidean space ${\mathbb R}^d$ with finite measure $|Q|$.
By $L_q=L_q(Q)$, $2\le q<\infty$, we mean the space of integrable on $Q$ $d$-variable  functions $f(t)$
equipped with the norm
$$
\|f\|_{q}:=
\bigg(\int_{Q} |f(t)|^q \, dt \bigg)^{\frac{1}{q}} .
$$
In $L_2$ the inner product is introduced in the standard way
$$
\langle f, g\rangle=\int_{Q} f(t) g(t) \ d t .
$$
By $C=C(Q)$ we mean the space of continuous on $Q$ functions.

In addition, we will make the following changes to the formulation of the optimal recovery problem (see Section 1).
Everywhere below, by $H$ we will mean $L_2(Q)$.
Further, let the sequences $\{\mu_k\}$ (\ref{mu}) and $\{\xi_k\}$ (\ref{xi}) are exponentially increasing,
the sequence $\{\frac{\mu_k}{\xi_k}\}$ is monotone decreasing and such that
for some constants $c_1, c_2>0$ and any $N$ it holds
$$
\Big(\sum_{k=0}^{N} \mu_k^2 \Big)^{1/2} \le c_1 \mu_N ,
$$
$$
\Big(\sum_{k=N}^{\infty} \frac{\mu_k^2}{\xi_k^2} \Big)^{1/2} \le c_2 \frac{\mu_N}{\xi_N} .
$$
Moreover, we suppose that there is a constant $\chi>0$ such that
for elements $w_k(t)$ of  orthonormal basis in $L_2$ it holds
$$
\|w_k\|_C \le \chi , \qquad k=0,1,2, \ldots .
$$
For any $2\le q \le \infty$ we will study the quantities
$$
\varepsilon_\delta (A, W, G, \Psi, L_q) :=
\sup_{\substack{f\in W, \\ \|f\|_{W}\leq 1}}\quad
\sup_{\substack{f^\delta\in L_2, \\ \|f-f^\delta\|_{2}\leq \delta}}\quad
\|Af - \Psi(G(f^\delta))\|_{q} ,
$$
$$
R_{N,\delta} (A, W, L_q) :=
\inf_{\substack{G: L_2\to {\mathbb R}^N, \\ \rm{(continuous\, map)}}}\quad
\inf_{\substack{\Psi: {\mathbb R}^N\to L_q \\ \rm{(arbitrary\, map)}}}\quad
\varepsilon_\delta (A, W, G, \Psi, L_q) ,
$$
where $L_\infty(Q)=C(Q)$.

%It holds

\vskip 2mm

\begin{theorem}  \label{th10}
Let $2\le q\le \infty$. For any $\delta$ it holds
$$
R_\delta (A, W, L_q) \ge  |Q|^{1/q-1/2} \frac{\mu_{N_\delta}}{\xi_{N_\delta}}
$$
and for any $\delta, N$ we have
$$
R_{N,\delta} (A, W, L_q) \ge |Q|^{1/q-1/2} \max \Big\{\frac{\mu_{N_\delta}}{\xi_{N_\delta}}, \frac{\mu_{N}}{\xi_{N}}\Big\} ,
$$
where $N_\delta$ is defined according to (\ref{N_delta}).
\end{theorem}

\begin{proof}
As in Theorem \ref{th1} we choose
$$
f_* = \xi_{N_\delta}^{-1} w_{N_\delta} .
$$
Then $f_*\in W$ with
$$
\|f_*\|_W = 1,\quad \|f_*\|_2 = \xi_{N_\delta}^{-1},\quad \|Af_*\|_2 = \frac{\mu_{N_\delta}}{\xi_{N_\delta}} .
$$
Repeating the reasoning from the proof of Theorem \ref{th1}, we obtain
$$
{R}_\delta (A, W, L_q) \ge |Q|^{1/q-1/2} \|Af_*\|_2
= |Q|^{1/q-1/2} \frac{\mu_{N_\delta}}{\chi_{N_\delta}} .
$$

When proving the second part of the lower estimate, we also follow the reasoning from Theorem \ref{th1}.
As a result, we have
$$
\varepsilon_\delta (A, W, G, \Psi, L_q) \ge \|Af_x\|_{q}
\ge |Q|^{1/q-1/2} \|Af_x\|_{2} \ge |Q|^{1/q-1/2}
\frac{\mu_{N}}{\xi_{N}} .
$$
Thus, Theorem \ref{th10} is completely proven.
\end{proof}

\vskip 2mm

\begin{theorem}  \label{th11}
For any $2\le q\le \infty$, $\delta$ and $N$ the recovering error of the truncation method (\ref{tm1}), (\ref{tm2})
is estimated as
$$
\varepsilon_\delta (A, W, \tilde G_N, \tilde \Psi_N, L_q) \le
\max\{c_1, c_2\} |Q|^{1/q} \chi \Big(\frac{\mu_N}{\xi_N} + \delta \mu_{N-1}\Big) .
$$
\end{theorem}

\begin{proof}
The proof of Theorem \ref{th11} repeats the line of reasoning from Theorem \ref{th2}.
It is only necessary to take into account that
$$
\|\sum_{k=N}^\infty \langle f , w_k\rangle \mu_k w_k\|_q
\le |Q|^{1/q} \|\sum_{k=N}^\infty \langle f , w_k\rangle \mu_k w_k\|_C
\le |Q|^{1/q} \chi \sum_{k=N}^\infty |\langle f , w_k\rangle| \mu_k
$$
$$
\le |Q|^{1/q} \chi \Big(\sum_{k=N}^\infty |\langle f , w_k\rangle|^2 \xi_k^2\Big)^{1/2}
\Big(\sum_{k=N}^\infty \frac{\mu_k^2}{\xi_k^2}\Big)^{1/2}
\le |Q|^{1/q} c_2 \chi \frac{\mu_N}{\xi_N} ,
$$
$$
\|\sum_{k=0}^{N-1} \langle f - f^\delta , w_k\rangle \mu_k w_k\|_q
\le |Q|^{1/q} \|\sum_{k=0}^{N-1} \langle f - f^\delta , w_k\rangle \mu_k w_k\|_C
\le |Q|^{1/q} \chi \sum_{k=0}^{N-1} |\langle f - f^\delta, w_k\rangle| \mu_k
$$
$$
\le |Q|^{1/q} \chi \Big(\sum_{k=0}^{N-1} |\langle f - f^\delta, w_k\rangle|^2\Big)^{1/2}
\Big(\sum_{k=0}^{N-1} \mu_k^2\Big)^{1/2}
\le |Q|^{1/q} c_1 \chi \delta \mu_{N-1} .
$$
This implies the statement of Theorem.
\end{proof}

The combination of Theorems \ref{th10}, \ref{th11} gives

\vskip 2mm

\begin{theorem}  \label{th12}
For $\xi_k$ (\ref{slowly}), any $2\le q\le \infty$, $\delta$ and $N$ such that $\frac{1}{\xi_N}\asymp \delta$
it holds
$$
{R}_\delta (A, W, L_q) \asymp
R_{N,\delta} (A, W, L_q) \asymp \frac{\mu_{N}}{\xi_{N}} .
$$
The indicated order-optimal estimates are achieved by the truncation method (\ref{tm1}), (\ref{tm2}).
\end{theorem}

%\vskip 2mm

%\begin{remark} {\rm
%Here, as in Theorem \ref{th9}, for "slowly" increasing sequences $\{\xi_k\}$, where
%$\{\frac{\xi_k}{e^{\alpha k}}\}$, $\alpha>0$, is a non-increasing sequence,
%the condition $\delta=\xi_N^{-1}$  can be replaced by a weaker condition $\delta\asymp\xi_N^{-1}$.
%}
%\end{remark}

\section{Optimal recovery in the numerical differentiation and the backward parabolic equation}  \label{ND&BPE}

Let's back to the problem of numerical differentiation (\ref{prex1}) (see Example \ref{ex1}).
From Theorem \ref{th12} it follows

\begin{theorem} \label{th13}
Let $\gamma>2$ in (\ref{prex1}) and $2\le q\le \infty$.
Then for any $\delta$ and $N$ such that
$$
N \asymp \Big(1/\delta\Big)^{1/\gamma}
$$
it holds true
$$
{R}_\delta (A_1, W, L_q) \asymp
R_{N,\delta}(A_1, W, L_q) \asymp
\delta^{(\gamma-2)/\gamma}
\asymp N^{-\gamma+2} .
$$
The indicated order-optimal estimates are achieved by the truncation method (\ref{tm1}), (\ref{tm2}).
\end{theorem}

\vskip 2mm

\begin{remark}\rm
Earlier the problem of optimal recovery of derivatives of periodic functions with inaccurate input data was considered in
\cite{ZT}, \cite{SS}. Note that the orders of error in the best approximation of the derivatives are the same in these works,
including this study. However, in \cite{SS} the input data was understood to be the perturbed values of the Fourier coefficients. Therefore, Theorem \ref{th13} of this study generalizes result of \cite{SS} to the case of discrete information of arbitrary form.
On the other hand, Theorem \ref{th13} covers the corresponding result of \cite{ZT} for the case of a wider set of perturbed inputs.
\end{remark}

\vskip 2mm

Now consider the backward parabolic equation (\ref{prex2}) (see Example \ref{ex2}).
%In what follows we will assume that $\gamma>0$.
The next statement follows from Theorems \ref{th9} and \ref{th12}.

\begin{theorem} \label{th14}
Let $\lambda_k\asymp k^\gamma$, $0<\gamma\le 1$, in (\ref{prex2}) and $2\le q\le \infty$.
Then for any $t$, $0\le t< T$, any $\delta$ and $N$ such that
$$
N^{-\mu} e^{-\lambda_{N} T} \asymp \delta
$$
it holds true
$$
{R}_\delta (A_2, W, L_q) \asymp
R_{N,\delta}(A_2, W, L_q) \asymp
\delta^{(T-t)/T}
(\ln 1/\delta)^{-\frac{\mu t}{\gamma T}}
\asymp N^{-\mu} e^{-\lambda_{N} (T-t)} .
$$
For $t=T$ and any $N$ such that
 $$
N \asymp (\ln 1/\delta)^{1/\gamma}
 $$
it holds true
$$
{R}_\delta (A_2, W, L_q) \asymp
R_{N,\delta}(A_2, W, L_q)
\asymp (\ln 1/\delta)^{-\mu/\gamma}
\asymp N^{-\mu} .
$$
The indicated order-optimal estimates are achived by the truncation method (\ref{tm1}), (\ref{tm2}).
\end{theorem}

%\vskip 2mm

%The next statement follows from Theorems \ref{th9} and \ref{th12}.
%\begin{theorem} \label{th15}
%Let $\gamma>1$ in (\ref{prex2}), $2\le q\le \infty$ and $\delta=N_\delta^{-\mu} e^{-\lambda_{N_\delta}T}$.
%Then for any $t$, $0\le t< T$,
%it holds true
%$$
%R_{N_\delta,\delta}(A_2, W, L_q) \asymp
%\delta^{(T-t)/T}
%(\ln 1/\delta)^{-\frac{\mu t}{\gamma T}}
%\asymp N_\delta^{-\mu} e^{-\lambda_{N_\delta} (T-t)} ,
%$$
%and for $t=T$
%it holds true
%$$
%R_{N_\delta,\delta}(A_2, W, L_2)
%\asymp (\ln 1/\delta)^{-\mu/\gamma}
%\asymp N_\delta^{-\mu} .
%$$
%The indicated order-optimal estimates are implemented by the truncation method (\ref{tm1}), (\ref{tm2}).
%\end{theorem}

\vskip 2mm

\begin{remark}  \label{Rem1}
{\rm
It can be seen from Theorem \ref{th14} that the type of ill-posedness of the problem (\ref{prex2})
depends on the value of $t$.
If, for example, $t=T$, then the optimal order of accuracy for solving the problem is
$O((\ln 1/\delta)^{-\frac{\mu}{\gamma}})$.
Such problems in the theory of ill-posed problems are usually called severely ill-posed problems.
Further, in the case of $0< t< T$ the backward parabolic equation (\ref{prex2}) can be solved with the accuracy
$O(\delta^{(T-t)/T})$ in the power scale.
Such an error characterizes moderately ill-posed problems.
Finally, for $t=0$, we are dealing with a well-posed problem,
the optimal accuracy of the solution of which is $O(\delta)$.
}
\end{remark}

\vskip 2mm

%\begin{remark}  {\rm
%As follows from Theorems \ref{th12}, \ref{th13}, \ref{th14}
%no mapping from $L_2$ into $L_q$ provides a higher order of accuracy than method (\ref{tm1}), (\ref{tm2}).
%}
%\end{remark}

%\vskip 2mm

\begin{remark}  \label{Rem2}
{\rm
For the first time, the optimal accuracy of solving backward parabolic problems for $t<T$ was found in \cite{T1998}.
Thus, Theorem  \ref{th14} of this study generalizes the result of \cite{T1998} to the case $t=T$.
In addition, Theorem  \ref{th14} generalizes the results of \cite{T1998} and \cite{MP}
to the case of any $L_q$, $2\le q\le \infty$.
}
\end{remark}

\vskip 2mm

\section*{Disclosure of Funding}
The second author has received funding through the MSCA4Ukraine project,
which is funded by the European Union (ID number 1232599).

\end{document}